\documentclass[10pt,reqno]{amsart}

\usepackage{amsmath}
\usepackage{amsfonts}
\usepackage{amssymb}
\usepackage{amsthm}
\usepackage{mathrsfs}
\usepackage{latexsym}

\usepackage{cite} 

\usepackage{esint}

\numberwithin{equation}{section}

\usepackage{graphicx,subfigure}

\usepackage{ifthen} 

\provideboolean{shownotes} 
\setboolean{shownotes}{true} 
\newcommand{\margnote}[1]{
\ifthenelse{\boolean{shownotes}}%
{\marginpar{\raggedright\tiny\texttt{#1}}}%
{}%
}

\newcommand{\hole}[1]{
\ifthenelse{\boolean{shownotes}}%
{\begin{center} \fbox{ \rule {.25cm}{0cm}
\rule[-.1cm]{0cm}{.4cm} \parbox{.85\textwidth}{\begin{center}
\texttt{#1}\end{center}} \rule {.25cm}{0cm}}\end{center}}
{}
}

\usepackage[colorlinks=true,urlcolor=blue,
citecolor=red,linkcolor=blue,linktocpage,pdfpagelabels,
bookmarksnumbered,bookmarksopen]{hyperref}

\theoremstyle{plain}

\newtheorem{lemma}{Lemma}[section]
\newtheorem{theorem}[lemma]{Theorem}
\newtheorem{proposition}[lemma]{Proposition}

\theoremstyle{definition}
\newtheorem{remark}[lemma]{Remark}
\newtheorem{definition}[lemma]{Definition}

\theoremstyle{remark}

\newcommand{\Id}{\mathrm{Id}}

\newcommand{\F}{\mathbb{F}}
\newcommand{\R}{\mathbb{R}}
\newcommand{\C}{\mathbb{C}}
\newcommand{\Z}{\mathbb{Z}}
\newcommand{\N}{\mathbb{N}}
\newcommand{\bbS}{\mathbb{S}}

\newcommand{\bp}{\overline{p}}
\newcommand{\bq}{\overline{q}}

\newcommand{\teta}{\widetilde{\eta}}
\newcommand{\tzet}{\widetilde{\zeta}}

\newcommand{\cT}{{\mathcal{T}}}

\newcommand{\cL}{{\mathcal{L}}}

\newcommand{\vep}{\varepsilon}

\renewcommand{\Re}{\mathrm{Re}\,} 
\renewcommand{\Im}{\mathrm{Im}\,}

\newcommand{\ep}{\epsilon}

\newcommand{\rr}{\sqrt{r}}

\newcommand{\ess}{\sigma_\mathrm{\tiny{ess}}}
\newcommand{\ptsp}{\sigma_\mathrm{\tiny{pt}}}

\newcommand{\Ldper}{L^2_\mathrm{\tiny{per}}}

\newcommand{\Hdper}{H^2_\mathrm{\tiny{per}}}
\newcommand{\Htper}{H^3_\mathrm{\tiny{per}}}
\newcommand{\Hmper}{H^m_\mathrm{\tiny{per}}}

\newcommand{\<}{\langle}
\renewcommand{\>}{\rangle}

\makeatletter
\@namedef{subjclassname@2020}{\textup{2020} Mathematics Subject Classification}
\makeatother

\begin{document}

\title[Instability of periodic waves for the KdV-Burgers equation]{Instability of periodic waves for the Korteweg-de Vries-Burgers equation with monostable source}

\author[R. Folino]{Raffaele Folino}
 
\address{{\rm (R. Folino)} Departamento de Matem\'aticas y Mec\'anica\\Instituto de 
Investigaciones en Matem\'aticas Aplicadas y en Sistemas\\Universidad Nacional Aut\'onoma de 
M\'exico\\ Circuito Escolar s/n, Ciudad Universitaria, C.P. 04510\\Cd. de M\'{e}xico (Mexico)}

\email{folino@aries.iimas.unam.mx}

\author[A. Naumkina]{Anna Naumkina}

\address{{\rm (A. Naumkina)} Departamento de Matem\'aticas y Mec\'anica\\Instituto de 
Investigaciones en Matem\'aticas Aplicadas y en Sistemas\\Universidad Nacional Aut\'onoma de 
M\'exico\\ Circuito Escolar s/n, Ciudad Universitaria, C.P. 04510\\Cd. de M\'{e}xico (Mexico)}

\email{naumkinaanna75@gmail.com}

\author[R. G. Plaza]{Ram\'on G. Plaza}

\address{{\rm (R. G. Plaza)} Departamento de Matem\'aticas y Mec\'anica\\Instituto de 
Investigaciones en Matem\'aticas Aplicadas y en Sistemas\\Universidad Nacional Aut\'onoma de 
M\'exico\\ Circuito Escolar s/n, Ciudad Universitaria, C.P. 04510\\Cd. de M\'{e}xico (Mexico)}

\email{plaza@aries.iimas.unam.mx}

\begin{abstract}
In this paper, it is proved that the KdV-Burgers equation with a monostable source term of Fisher-KPP type has small-amplitude periodic traveling wave solutions with finite fundamental period. These solutions emerge from a subcritical local Hopf bifurcation around a critical value of the wave speed. Moreover, it is shown that these periodic waves are spectrally unstable as solutions to the PDE, that is, the Floquet (continuous) spectrum of the linearization around each periodic wave intersects the unstable half plane of complex values with positive real part. To that end, classical perturbation theory for linear operators is applied in order to prove that the spectrum of the linearized operator around the wave can be approximated by that of a constant coefficient operator around the zero solution, which intersects the unstable complex half plane.
\end{abstract}



\maketitle

\setcounter{tocdepth}{1}



\section{Introduction}
\label{secintro}

The Korteweg-de Vries-Burgers (KdVB) equation,
\[
u_t + \alpha uu_x - \mu u_{xx} + \beta u_{xxx} = 0,
\]
with $\mu > 0$, $\beta > 0$, was first derived by Su and Gardner \cite{SuGa69} for a wide class of Galilean invariant physical systems under the weak nonlinearity and long wavelength approximations, resulting into a model equation which underlies a balance between both dispersive and dissipative effects. It has been used to describe the approximate behavior of many physical phenomena such as weakly nonlinear plasma waves with electron inertial effects \cite{HuPN72}, undular bores in shallow water \cite{Jo72}, cascading processes in turbulence theory \cite{LiuLiu92}, wave propagation in an elastic tube filled with a viscous fluid \cite{Jo70}, the flow of liquids containing gas bubbles \cite{Wijn72}, waves in electronegative plasma \cite{Gao22,AAA14} and the description of gravity waves in atmospheric dynamics \cite{Jo97}, among many others. A combination of the (viscous) Burgers \cite{Bur48} and the Korteweg-de Vries \cite{KdV1895} equations, the KdVB model is perhaps the simplest physically relevant scalar equation containing a nonlinear transport or convection term, $uu_x$, a dissipation or viscosity term, $u_{xx}$, and a dispersive term, $u_{xxx}$, into one single equation. In all the aforementioned applications, the study of traveling wave solutions is a fundamental issue. The literature on nonlinear wave propagation for the KdVB equation is very extensive and we do not pretend to make a review here. The reader is referred to the classical review paper by Jeffrey and Kakutani \cite{JeKa72} and to the references cited therein.

In this paper, we consider the KdVB equation with a monostable source term, also known as the Korteweg-de Vries-Burgers-Fisher (KdVBF) equation \cite{Koc20},
\begin{equation}
\label{KdVBF}
u_t + \alpha uu_x - \mu u_{xx} + \beta u_{xxx} = ru(1-u),
\end{equation}
describing the dynamics of a real, scalar unknown quantity $u = u(x,t)$, depending on space, $x \in \R$, and time, $t > 0$, and where $\alpha, \beta, \mu, r \geq 0$ are non-negative constants. Since we are interested in the interplay between dissipation, dispersion, nonlinear transport and reaction, we assume that all physical constants are positive, $\alpha, \beta, \mu, r > 0$. Under this assumption and without loss of generality, we normalize the model by making the transformations
\[
x \to \frac{\mu}{\beta} x, \quad t \to \frac{\mu^3}{\beta^2} t, \quad \alpha \to \frac{\mu^4}{\beta^3} \alpha, \quad r \to \frac{\beta^2}{\mu^3} r,
\]
in order to recast equation \eqref{KdVBF} as
\begin{equation}
\label{KdVBFn}
u_t + \alpha uu_x + u_{xxx} = u_{xx} + ru(1-u),
\end{equation}
for some fixed constants $\alpha > 0$ and $r > 0$; this is the model equation that we work with for the rest of the paper. 

The KdVBF equation \eqref{KdVBF} was recently introduced by Ko\c{c}ak \cite{Koc20}, who proved the existence of traveling waves (such as solitons, kink and anti-kink waves) under the presence of the reaction term
\[
f(u) = ru(1-u),
\]
also known as a logistic reaction function. This particular form of the reaction term was considered by Fisher \cite{Fis37} and by Kolomogorov, Petrovsky and Piskunov \cite{KPP37} in the context of reaction-diffusion equations modelling invasion fronts, and it is often used to model dynamics of populations with limited resources which saturate into a stable equilibrium point associated to an intrinsic carrying capacity (in this case, the equilibrium state $u = 1$). Although the nomenclature is not uniform, it is often known under the names of monostable, logistic or Fisher-KPP reaction term. It is one of the most widely used production mechanisms giving rise to wave propagation and determining their dynamics. Other studies have also considered the KdVB equation with sources: Gao \cite{Gao22}, for instance, derived the KdVB equation with a linear reaction term (a damping source) in the study of waves in electronegative plasma. We conjecture that the KdVBF equation might be derived in the long wavelength and weakly nonlinear approximations from a physically significant system, albeit such derivation is, as far as we know, still unknown.

The purpose of this paper is to contribute to the analysis of nonlinear waves for the KdVBF equation \eqref{KdVBF}. We study for the first time the problem of existence of periodic waves. Instead of applying the common methods based on exploiting symmetries, or on solving the differential profile equations on a case-by-case basis, we prove the existence of periodic waves via application of local bifurcation methods \cite{ZLZ13,AlPl21,AMP22,CheDu22}. Profiting from the presence of the unstable equilibrium point of the reaction function, we prove that a family of small-amplitude periodic waves emerges from a subcritical local Hopf bifurcation around a critical value of the wave speed; this is the first result of the paper, see Theorem \ref{thmexist} below. These waves have small amplitude and bounded fundamental period. The amplitudes are of order of the square root of the distance between the speed and the critical speed. The instability of the equilibrium point of the reaction is responsible for the existence of the waves and for the change of stability of the equilibrium point as the wave speed crosses the bifurcation critical value. Some numerical approximations illustrate the emergence of this family of waves.

Next, we study the stability of these bounded periodic waves as solutions to the nonlinear evolution equation \eqref{KdVBFn}. As it is customary, we first linearize the equation around each periodic wave and study the resulting spectral problem. In the case of periodic waves, the linearized operator has periodic coefficients, leading to the concept of the \emph{Floquet spectrum}, see \cite{KaPro13,JMMP14,Grd1}. We then recast the problem of locating the continuous or Floquet spectrum as a point spectral problem on an appropriate periodic space via a Bloch-type transformation. Since the waves have small amplitude, we follow previous analyses \cite{AlPl21,AMP22,CheDu22,KDT19} and prove that the spectrum can be approximated by the spectrum of a constant coefficient operator around the zero solution. For that purpose, a key ingredient is to show that the perturbed problem encompasses relatively bounded perturbations. The analysis is based on recasting the spectral equations for the Bloch-type operators as a perturbation problem, where the perturbation parameter is precisely the distance between the speed and the critical speed. We then apply the classical perturbation theory for linear operators (cf. \cite{Kat80,HiSi96}) to show that the unstable point eigenvalue of the unperturbed operator splits into neighboring curves of Floquet spectra of the underlying small amplitude waves, proving in this fashion the spectral instability result, which is the main result of the paper, see Theorem \ref{mainthm} below. Up to our knowledge, this is the first analysis of existence and stability properties of periodic waves for the KdVBF equation \eqref{KdVBFn}.

\subsection*{On notation}
We denote the real and imaginary parts of a complex number $\lambda \in \C$ by $\Re\lambda$ and $\Im\lambda$, respectively, as well as complex conjugation by ${\overline{\lambda}}$. 
Standard Sobolev spaces of complex-valued functions on the real line will be denoted as $L^2(\R)$ and $H^m(\R)$, with $m \in \N$, endowed with the standard inner products and norms. For any fundamental period $L > 0$, we denote by $\Ldper([0,L])$ the Hilbert space of complex $L$-periodic functions in $L^2_\mathrm{\tiny{loc}}(\R)$ satisfying
$u(x + L) = u(x)$ a.e. in $x$ and with inner product and norm given by
\[
\< u, v \>_{\Ldper} = \int_0^{L} u(x) \overline{v(x)} \, dx, \qquad \| u \|^2_{\Ldper} = \< u, u \>_{\Ldper}.
\]
For any $m \in \N$, the periodic Sobolev space $\Hmper([0,L])$ will denote the set of all functions $u \in \Ldper([0,L])$ whose weak derivatives up to order $m$ belong to $\Ldper([0,L])$. Their standard inner product and norm are given by $\< u,v \>_{\Hmper} = \sum_{j=0}^m \< \partial_x^j u, \partial_x^j v\>_{\Ldper}$ and $\|u\|_{\Hmper}^2 = \< u,u\>_{\Hmper}$, respectively.

%
%
%
%
%
%
%
%
%

\section{Small-amplitude periodic waves}
\label{secexistence}

This section is devoted to prove that small amplitude periodic wave solutions to equation \eqref{KdVBFn} do exist. These waves emerge from a Hopf bifurcation thanks to the the presence of an unstable equilibrium point of the reaction function.

\subsection{Existence of periodic waves}

A periodic wave is a solution to \eqref{KdVBFn} of the form
\begin{equation}
\label{TWS}
u(x,t) = \varphi(x - ct),
\end{equation}
where the profile function, $\varphi = \varphi(\xi)$, $\varphi : \R \to \R$, is at least of class $C^3$ and periodic, with fundamental period $L > 0$, that is, $\varphi(\xi +L) = \varphi(\xi)$ for all $\xi \in \R$. Here $\xi = x-ct$ denotes the Galilean variable of translation, and $c \in \R$ is the wave speed. Upon substitution of the \emph{ansatz} \eqref{TWS} into \eqref{KdVBFn} one obtains the differential profile equation,
\begin{equation}
\label{profileq}
-c \varphi' + \alpha \varphi \varphi' + \varphi''' = \varphi'' + r\varphi(1-\varphi),
\end{equation}
where $' = d/d\xi$.

In order to rewrite \eqref{profileq} as a first order system, let us define
\[
\Phi = \begin{pmatrix} \phi_1 \\ \phi_2 \\ \phi_3 \end{pmatrix} := \begin{pmatrix} \varphi \\ \varphi' \\ \varphi'' \end{pmatrix},
\]
so that
\begin{equation}
\label{firstordsyst}
\Phi' = F(\Phi) := \begin{pmatrix} \phi_2 \\ \phi_3 \\ \phi_3 + r\phi_1(1-\phi_1) - \alpha \phi_1 \phi_2 + c \phi_2 \end{pmatrix}.
\end{equation}
Notice that $F \in C^\infty(\R^3;\R^3)$ and the only equilibrium points in $\R^3$ of system \eqref{firstordsyst} are $P_0 = (0,0,0)$ and $P_1 = (1,0,0)$. The Jacobian of $F$ is given by
\[
A(\Phi) := D_\Phi F (\Phi) = \begin{pmatrix} 0 & 1 & 0 \\ 0 & 0 & 1 \\ r(1-2\phi_1)-\alpha \phi_2 & c - \alpha \phi_1 & 1 \end{pmatrix} \in C^\infty(\R^3;\R^{3\times 3}).
\]
The Jacobian evaluated at $P_0$ is
\[
A(P_0) = \begin{pmatrix} 0 & 1 & 0 \\ 0 & 0 & 1 \\ r & c  & 1 \end{pmatrix}.
\]
Therefore, its characteristic equation is given by
\begin{equation}\label{charA0}
Q(\lambda,c) := \det \big(\lambda I - A(P_0)\big) = \lambda^2(\lambda -1) - \lambda c - r = 0.
\end{equation}
For fixed $r > 0$, the complex roots of equation \eqref{charA0} depend on $c \in \R$ and are denoted as $\lambda_j = \lambda_j(c)$, $j =1,2,3$. In our setting, the speed value $c \in \R$ will play the role of the bifurcation parameter. Notice that, evaluated at the critical speed
\begin{equation}
\label{critspeed} c_0 := -r, 
\end{equation}
the characteristic polynomial is $Q(\lambda, -r) = (\lambda-1)(\lambda^2 + r)$ and its roots are
\[
\lambda_1(c_0) = 1, \qquad \lambda_2(c_0) = i \sqrt{r}, \qquad \lambda_3(c_0) = -i \sqrt{r}.
\]

If we denote $\eta(c) = \Re \lambda(c)$ and $\zeta(c) = \Im \lambda(c)$ as the real and imaginary parts of any root $\lambda(c)$ of \eqref{charA0}, respectively, then upon substitution into \eqref{charA0} we have
\[
(\eta + i \zeta)^2 (\eta -1 + i \zeta) - c (\eta + i \zeta) - r = 0,
\]
for all $c \in \R$. Taking the real and imaginary parts of last equation yields $G(\eta,\zeta,c) = 0$, with $G : \R^3 \to \R^2$, $G = (G_1, \, G_2)$, and
\[
\begin{aligned}
G_1(\eta,\zeta,c) &= (\eta^2 - \zeta^2)(\eta-1) - 2 \eta \zeta^2 -c \eta - r = 0,\\
G_2(\eta,\zeta,c) &= \zeta(\eta^2 - \zeta^2) + 2 \eta \zeta (\eta-1) - c \zeta = 0.
\end{aligned}
\]

Let us denote $(\eta_0, \zeta_{0j}, c_0) := (\Re \lambda_j(c_0), \Im \lambda_j(c_0), c_0) = (0, (-1)^j \sqrt{r}, -r)$, for $j = 2,3$. Then, clearly, $G_1(\eta_0, \zeta_{0j}, c_0) = r-r = 0$ and 
\[
G_2(\eta_0, \zeta_{0j}, c_0) = -(-1)^j r^{3/2} + (-1)^j r^{3/2} = 0.
\]
Moreover, the Jacobian matrix of $G$ with respect to the $(\eta,\zeta)$-variables is
\[
D_{(\eta,\zeta)} G = \begin{pmatrix} 3(\eta^2 - \zeta^2) - 2 \eta -c & 2\zeta - 6 \eta \zeta \\ -2\zeta + 6 \eta \zeta & 3(\eta^2 - \zeta^2) - 2 \eta -c \end{pmatrix}.
\]
%
Hence, evaluating this Jacobian at $(\eta_0, \zeta_{0j}, c_0)$ we obtain
\[
J_0 := \big( D_{(\eta,\zeta)} G \big) |_{(\eta_0, \zeta_{0j}, c_0)} = \begin{pmatrix} -2r & (-1)^j 2 \sqrt{r} \\ -(-1)^j 2 \sqrt{r} & -2r \end{pmatrix}, \qquad j =2,3.
\]
Notice that $\det J_0 = 4r(r+1) > 0$. Henceforth, by the Implicit Function Theorem there exist local functions
\[
(\teta_j, \tzet_j)(c) = (\Re \lambda_j(c), \Im \lambda_j(c)),
\]
defined on a neighborhood of $c_0$ such that $\teta_j(c_0) = \eta_0$, $\tzet_j(c_0) = \zeta_{0j}$, and there holds $G(\teta_j(c), \tzet_j(c), c) = 0$ for all $c$ in such neighborhood of  $c_0$. Moreover, $(\teta_j, \tzet_j)$ are of class $C^1$ and their derivatives are given by
\[
\frac{d}{dc} \begin{pmatrix} \teta_j \\ \tzet_j \end{pmatrix} = - \left. \left[ (D_{(\eta,\zeta)} G)^{-1} \begin{pmatrix} \partial_c G_1 \\ \partial_c G_2 \end{pmatrix} \right] \right|_{(\teta_j(c), \tzet_j(c), c)}.
\]
Since $\partial_c G_1 = - \eta$ and $\partial_c G_2 = - \zeta$ we readily obtain
\[
\begin{aligned}
\left. {\frac{d}{dc} \begin{pmatrix} \teta_j \\ \tzet_j \end{pmatrix}} \right|_{c_0} =  J_0^{-1} \begin{pmatrix} \eta_0 \\ \zeta_{0j} \end{pmatrix} 
&= \frac{1}{4r(r+1)} \begin{pmatrix} -2r & -(-1)^j 2 \sqrt{r} \\ (-1)^j 2 \sqrt{r} & -2r \end{pmatrix} \begin{pmatrix} 0 \\ (-1)^j \sqrt{r} \end{pmatrix} \\
& = \frac{-1}{2(r+1)} \begin{pmatrix} 1 \\ (-1)^j \sqrt{r}\end{pmatrix}, \qquad j =2,3.
\end{aligned}
\]
This yields
\begin{equation}
\label{derrelam}
d_0 := \left. \frac{d}{dc} \big(\Re \lambda_j(c)\big) \right|_{c_0} =  \frac{-1}{2(r+1)} < 0, \qquad j=2,3.
\end{equation}

To sum up, for fixed physical parameter values, $\alpha > 0$ and $r > 0$, the first order system \eqref{firstordsyst} is, with a slight abuse of notation, of the form $\Phi' = F(\Phi,c)$ where $c \in \R$ is interpreted as a bifurcation parameter, $F$ is smooth in both variables and it satisfies that, when evaluated at $(P_0, c_0)$ (where $c_0 = -r$ is the critical bifurcation parameter), the Jacobian $A_0 := (D_\Phi F)|_{(P_0,c_0)}$ has two simple purely imaginary eigenvalues, $\lambda_2(c_0)$ and $\lambda_3(c_0)$, with $\lambda_j(c_0) = (-1)^j i \, \omega_0$, for $j=2,3$ and $\omega_0 := \sqrt{r} > 0$, and no other eigenvalues with non zero real parts. Moreover, the non-degeneracy condition \eqref{derrelam} holds, yielding $\Re \lambda_j(c) < 0$ for $c > c_0$. Finally, it can be verified that the first Lyapunov coefficient is strictly positive, $a_0 = \ell_1(0) > 0$, and the genericity condition also holds (see Appendix \ref{apendice}).

Upon application of Hopf's bifurcation Theorem (see Theorem 3.15 in \cite{MaMcC76}, p. 81, or Theorem 3.4.2 in \cite{GuHo83}, p. 151) we conclude that a local Hopf bifurcation occurs at $c = c_0$ and we deduce the existence of $0 < \ep_0 \ll 1$ sufficiently small such that the system has a unique (up to translations) family of closed periodic orbit solutions, $\Phi_\ep = \Phi_\ep(\xi)$, for each value of $0 < \ep < \ep_0$. This family is parametrized by speed values of the form
\[
c(\ep) = c_0 + \ep = -r + \ep.
\]
The orbits emerge for values $c > c_0$ because $a_0 d_0 < 0$. Each periodic orbit has a fundamental period given by $L_\ep = 2\pi/\omega_0 + O(\ep) = 2\pi/\sqrt{r} + O(\ep)$ and amplitude growing like $|\Phi_\ep| = O(\sqrt{\ep})$. Moreover, from the regularity of the vector field, it follows that $\Phi_\ep \in C^\infty$. Identifying the first component of each of these periodic orbits as a periodic traveling wave for equation \eqref{KdVBFn}, we have thus proved the following existence result.

\begin{theorem}[existence of small amplitude periodic waves]
\label{thmexist}
For any fixed parameter values $r, \alpha > 0$, there exist $\ep_0 > 0$ sufficiently small and a critical speed value $c_0 = -r$ such that the KdV-Burgers-Fisher equation \eqref{KdVBFn} has a family of smooth periodic traveling wave solutions of the form $u(x,t) = \varphi^\ep(x - c(\ep)t)$, indexed by $\ep\in(0,\ep_0)$,  with fundamental period 
\begin{equation}
\label{period}
L_\ep = \frac{2\pi}{\sqrt{r}} + O(\ep),
\end{equation}
and with amplitude growing like
\begin{equation}
\label{amplitude}
|\varphi^\ep|, |(\varphi^\ep)'| = O(\sqrt{\ep}).
\end{equation}
For any $\ep \in (0,\ep_0)$, the speed of propagation is given by $c(\ep) = -r + \ep$,  and satisfies $c(\ep) \to c_0$ as $\ep \to 0^+$. 
Finally, the periodic wave $\varphi^\ep$ is unique up to translations, for each fixed $\ep \in (0,\ep_0)$.
\end{theorem}

\begin{remark}
A few remarks are in order. Notice that the linearization at the rest point $P_0$ passes from a unstable hyperbolic focus with $\Re \lambda_j > 0$, $j = 1,2,3$, for $c < c_0$ to having two stable eigenvalues $\Re \lambda_j < 0$, $j = 2,3$ and one unstable $\Re \lambda_1 > 0$ for values $c > c_0$. Therefore, the limit cycle for $c > c_0$ is unstable. When a stable limit cycle surrounds an unstable equilibrium point, the bifurcation is called a supercritical Hopf bifurcation. If the limit cycle is unstable and surrounds a stable equilibrium point, then the bifurcation is called a subcritical Hopf bifurcation (cf. \cite{MaMcC76,Strg15}). Here the equilibrium point is stable only on a two-dimensional normally hyperbolic invariant stable manifold, although there is a permanent unstable eigenvalue ($\lambda_1(c_0) = 1$) near the critical bifurcation point\footnote{Actually, if there is at least one eigenvalue with positive real part then the periodic orbit is always unstable with respect to the full flow in $\R^3$; see Remark 3, p. 20, in Hassard \emph{et al.} \cite{HsKzW81}.}. Hence, this corresponds to a \emph{subcritical} Hopf bifurcation, even though it takes place for values $c > c_0$ (here the direction of the bifurcation is reversed because of the sign of the derivative in \eqref{derrelam}), and the emerging periodic orbits are unstable. This is consistent with the sign of the first Lyapunov coefficient, which is positive. Of course, this notion of stable/unstable periodic orbit refers to the standard concept from dynamical systems theory: an orbit is stable as a solution to system \eqref{firstordsyst} for a constant value of $c$ if any other nearby solution to the system with the same $c$ tends to the orbit under consideration. This definition of stability is completely unrelated to the concept of \emph{spectrally stable periodic wave}, which refers to the dynamical stability of the traveling wave as a solution to the evolution PDE; see \cite{AlPl21,AMP22} for further information and discussions. 
\end{remark}

\subsection{Numerical approximation}

In order to illustrate the emergence of small amplitude periodic waves for the KdVBF equation \eqref{KdVBFn}, let us fix the physical parameters as $r = \alpha = 1$. Then the dynamical system \eqref{firstordsyst} in $\R^3$  now reads
\begin{equation}
\label{firstordsyst2}
\begin{pmatrix} \phi_1 \\ \phi_2 \\ \phi_3\end{pmatrix} ' = \begin{pmatrix} \phi_2 \\ \phi_3 \\ \phi_3 + \phi_1(1-\phi_1) -  \phi_1 \phi_2 + c(\ep) \phi_2 \end{pmatrix},
\end{equation}
where the speed, or bifurcation parameter, is given by $c(\ep) = c_0 + \ep = -1 + \ep$ and $\ep$ takes values in a neighborhood of zero. As before, $' = d/d\xi$, where $\xi = x - c(\ep)t$ denotes the translation variable, once the value of the speed has been fixed. Figure \ref{figFlow} shows the flow in the phase space of the solutions system \eqref{firstordsyst2} for $c(\ep) = -1 \pm \ep$ with $\ep = 0.001$. To that end, we used the standard tools provided by \textsc{Mathematica}\copyright. For instance, Figure \ref{figFlowneg} depicts the flow when $c = -1 -0.001$, that is, before the bifurcation occurs. Notice that the origin is a hyperbolic saddle with both stable and unstable eigendirections. Figure \ref{figFlowpos} shows the flow in phase space of the system \eqref{firstordsyst2} for $c = -1 + 0.001$, right after the subcritical bifurcation occurs. The unique, small-amplitude periodic orbit that emerges as a result of such bifurcation is represented in red color.

\begin{figure}[t]
\begin{center}
\subfigure[$c=-1-0.001$]{\label{figFlowneg}\includegraphics[scale=.45, clip=true]{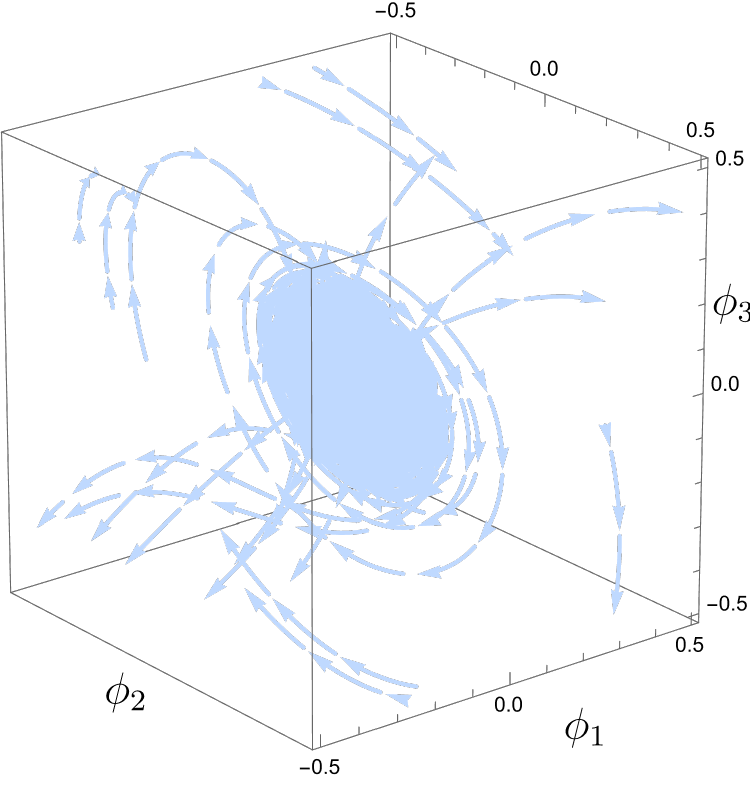}}
\subfigure[$c=-1+0.001$]{\label{figFlowpos}\includegraphics[scale=.45, clip=true]{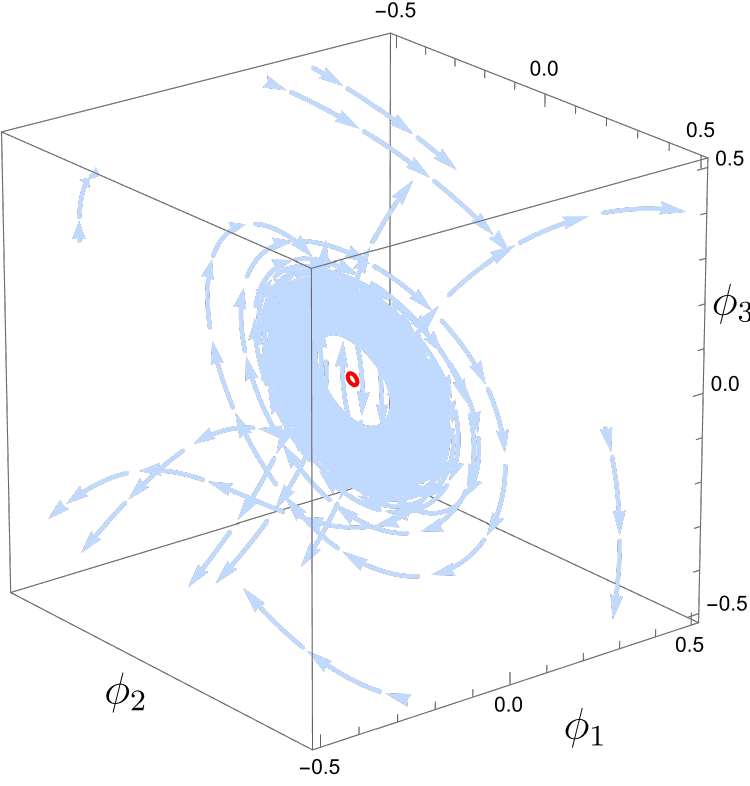}}
\end{center}
\caption{\small{Emergence of small-amplitude periodic waves for the KdVBF equation \eqref{KdVBFn} with $r = \alpha = 1$. Panel (a) shows the flow in the three dimensional phase space of system \eqref{firstordsyst2} for a speed value of $c(\ep) = -1 - \ep$ with $\ep = 0.001$. The arrows in light blue color show the direction of the vector field. The origin is a hyperbolic saddle with both stable and unstable eigendirections. Panel (b) shows the flow in phase space of system \eqref{firstordsyst2} for a speed value of $c(\ep) = -1 + \ep$ with $\ep = 0.001$, after the subcritical bifurcation occurs. As a result, a small amplitude periodic orbit (shown in red color) emerges around the origin (color online).}}\label{figFlow}
\end{figure}

It is to be observed that the small periodic orbit appearing in Figure \ref{figFlowpos} is, in fact, a numerical approximation. For that purpose, we used the \texttt{NDSolve} package, in which a standard explicit Runge-Kutta method was implemented with an initial point given by $(0,0.01,0)$. Figure \ref{figWaveZoom} shows a magnification of the flow in phase space for the speed value $c = -1 + 0.001$ (the same as in Figure \ref{figFlowpos}), together with the computed periodic orbit. Figure \ref{figWavephi} contains the plot of the first component of the approximated periodic orbit, $\phi_1$, as a function of the translation variable $\xi = x - c(\ep)t$.

\begin{figure}[t]
\begin{center}
\subfigure[$c=-1+0.001$]{\label{figWaveZoom}\includegraphics[scale=.45, clip=true]{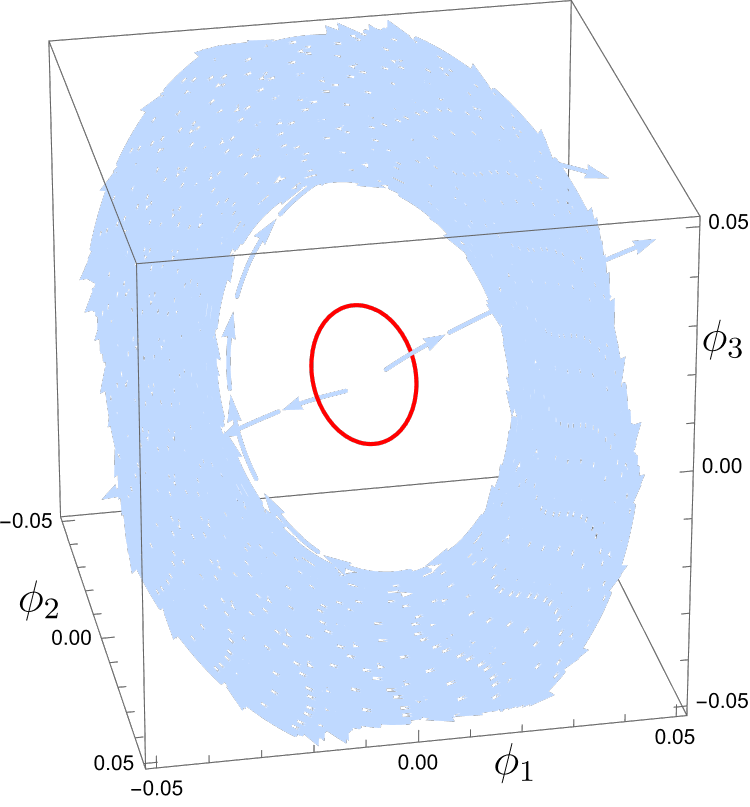}}
\subfigure[$\varphi^\ep(\xi) = \phi_1(\xi)$, $\xi = x - c(\ep)t$]{\label{figWavephi}\includegraphics[scale=.45, clip=true]{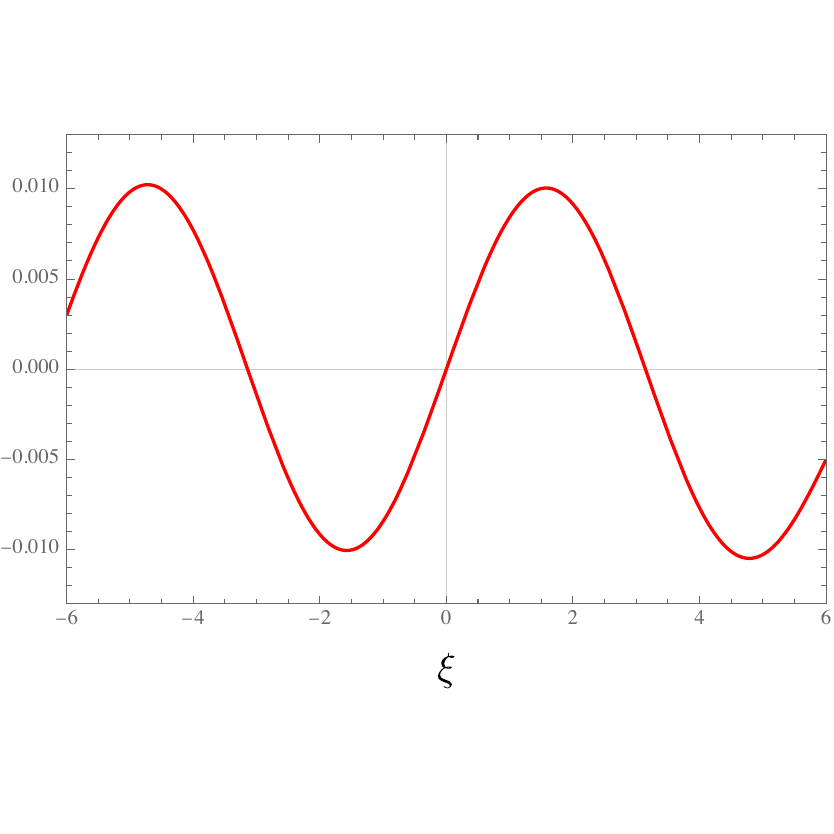}}
\end{center}
\caption{\small{Numerical approximation of the periodic orbit for $c = -1 + 0.001$. Panel (a) is a magnification of the flow Figure \ref{figFlowpos}, together with the numerical approximation of the emerging periodic orbit (in red color). Panel (b) shows the graph (in red) of the first component $\phi_1$ of the numerical approximation of the periodic orbit as a function of the Galilean variable of translation, $\xi = x - c(\ep)t$ (color online).}}\label{figWave}
\end{figure}


\section{The stability problem}

In this section we make precise the notion of spectral stability of the periodic waves under consideration and exhibit its equivalence to computing the Floquet spectrum of the linearized operator around the wave.

\subsection{Spectral stability}

Fix a parameter value $\ep \in (0, \ep_0)$ and consider the unique traveling wave solution to \eqref{KdVBFn}, $u(x,t) = \varphi^\ep(x - c(\ep)t) = \varphi^\ep(\xi)$, from Theorem \ref{thmexist}. Denoting $\xi = x - c(\epsilon)t$ as the translation variable, like before, we rewrite equation \eqref{KdVBFn} as
\begin{equation}\label{KdVBFn2}
u_t - c(\ep) u_{\xi} + \alpha uu_{\xi} + u_{\xi \xi \xi} = u_{\xi \xi } + ru(1-u).
\end{equation}
In other words, we recast the problem into a Galilean coordinate frame under which the profile $\varphi^\ep = \varphi^\ep (\xi)$ is now a stationary solution to \eqref{KdVBFn2} (thanks to equation \eqref{profileq}). Motivated by the notion of spatially localized, finite energy perturbations in the Galilean coordinate frame in which the periodic wave is stationary, let us consider solutions to \eqref{KdVBFn2} of the form
\[
u(\xi,t) = \varphi^\ep(\xi) + e^{\lambda t} v(\xi),
\]
where $\lambda \in \C$ and $v \in L^2(\R)$. Upon substitution and using the profile equation \eqref{profileq}, the spatial perturbation $v$ must be a solution to the following nonlinear equation
\[
\lambda v + \alpha vv_\xi + (\alpha \varphi^\ep(\xi) - c(\ep)) v_\xi + v_{\xi\xi\xi} = v_{\xi\xi} + \big( r(1-2\varphi^\ep(\xi)) - \alpha (\varphi^\ep)'(\xi)\big) v - r\varphi^\ep(\xi)v^2.
\]
Linearizing the last equation we obtain
\[
\lambda v = - v_{\xi \xi \xi } + v_{\xi \xi } - (\alpha \varphi^\ep(\xi ) - c(\ep)) v_\xi  + \big( r(1-2\varphi^\ep(\xi )) - \alpha (\varphi^\ep)'(\xi )\big) v,
\]
that is, we arrive at the following spectral problem for the linearized operator around the periodic wave,
\begin{equation}
\label{specprob}
\lambda v = \cL^\ep v, 
\end{equation}
where, for each $0 < \ep < \ep_0$, the linear operator $\cL^\ep$ is defined as
\begin{equation}
\label{defL}
\left\{
\begin{aligned}
\cL^\ep &: L^2(\R) \to L^2(\R),\\
D(\cL^\ep) &= H^3(\R),\\
\cL^\ep &:= - \partial_{\xi}^3 + \partial_{\xi}^2 + a_1^\ep(\xi) \partial_{\xi} + a_0^\ep(\xi) \Id,
\end{aligned}
\right.
\end{equation}
with coefficients
\begin{equation}\label{defaeps}
\begin{aligned}
a_1^\ep(\xi ) &:= c(\ep) - \alpha \varphi^\ep(\xi ),\\
a_0^\ep(\xi ) &:= r(1-2\varphi^\ep(\xi )) - \alpha (\varphi^\ep)'(\xi ),
\end{aligned}
\end{equation}
for all $\xi  \in \R$. Since the periodic wave is smooth and $L_\ep$-periodic, the coefficients $a_j^\ep$, $j=0,1$, and their derivatives are smooth, bounded and $L_\ep$-periodic. Notice that the operator $\cL^\ep$ is densely defined in $L^2(\R)$, with domain $D(\cL_\ep) = H^3(\R)$. Moreover, in view that the coefficients clearly satisfy $a_j^\ep \in W^{1,\infty}(\R)$, $j=0,1$, from Lemma 3.1.2 in \cite{KaPro13}, p. 40, we conclude that $\cL^\ep$ is a closed operator in $L^2(\R)$.

Let us now recall some definitions from classical spectral theory. Let $\cL : X \to Y$ be a densely defined, closed, linear operator with domain $D(\cL) \subseteq X$ and $X, Y$ Banach spaces. The resolvent set of $\cL$, denoted as $\rho(\cL)_{|X}$, is defined as the set of complex numbers $\lambda \in \C$ such that $\cL - \lambda$ is invertible and $(\cL - \lambda)^{-1}$ is a bounded operator. The complement of $\rho(\cL)_{|X}$ is what we call the $X$-spectrum of $\cL$ and we denote it as $\sigma(\cL)_{|X} = \C \backslash \rho(\cL)_{|X}$. Any complex number, $\lambda \in \C$, is said to belong to the point spectrum, $\ptsp(\cL)_{|X}$, if $\cL - \lambda$ is a Fredholm operator with index equal to zero and non-trivial kernel. Also, $\lambda$ belongs to the essential (or \emph{continuous}, in the periodic case) spectrum, $\ess(\cL)_{|X}$, provided that either $\cL - \lambda$ is not Fredholm or it is Fredholm with non-zero index. Clearly, both the point and essential spectra are subsets of $\sigma(\cL)_{|X}$. Moreover, since the operator is closed then there holds $\sigma(\cL)_{|X} = \ptsp(\cL)_{|X} \cup \ess(\cL)_{|X}$. Indeed, it suffices to verify that $\ptsp(\cL)^c_{|X} \cap \ess(\cL)^c_{|X} \subset \rho(\cL)$. But by definition, $\lambda \in \ptsp(\cL)^c_{|X} \cap \ess(\cL)^c_{|X}$ implies that $\cL - \lambda$ is Fredholm with index zero and trivial kernel. Therefore, the nullity (the dimension of $\ker(\cL - \lambda)$) and the deficiency (codimension of the range of $\cL - \lambda$) are both equal to zero, yielding $\mathrm{Ran}(\cL-\lambda) = Y$. Thus, $\cL - \lambda$ is invertible. Since $\cL - \lambda$ is closed,  this implies that $(\cL - \lambda)^{-1}$ is a bounded operator (see \cite{Kat80}, p. 167, Problem 5.21). Hence, we conclude that $\lambda \in \rho(\cL)$. It is to be observed that the point spectrum comprises discrete eigenvalues with finite (algebraic) multiplicity. The reader is referred to Kato \cite{Kat80} and Kapitula and Promislow \cite{KaPro13} for further information.

\begin{definition}[spectral stability]
\label{defspectstab} 
The periodic traveling wave solution, $\varphi^\ep$, to \eqref{KdVBFn} is \emph{spectrally stable} if the $L^2$-spectrum of the linearized operator around the wave $\cL^\ep$ satisfies
\[
\sigma(\cL^\ep)_{|L^2} \subset \{\lambda \in \C \, : \, \Re \lambda \leq 0\}.
\]
Otherwise, we say that it is \emph{spectrally unstable}.
\end{definition}

\begin{remark}
The choice of $L^2$ as a base space is motivated by the notion of spatially localized, finite energy perturbations in the translation coordinate frame in which the wave is stationary.
\end{remark}

\subsection{The Floquet spectrum}
\label{secFloqsp}

Since the coefficients of the differential operator $\cL^\ep$ are periodic, it is well known from Floquet theory that $\cL^\ep$ has no $L^2$-point spectrum and that $\sigma(\cL^\ep)_{|L^2} = \ess(\cL^\ep)_{|L^2}$ (see Lemma 3.3 in \cite{JMMP14}, or Lemma 59, p. 1487, in \cite{DunSch2}). Moreover, it is possible to parametrize the essential spectrum in terms of Floquet multipliers of the form $e^{i\theta} \in \bbS^1$, $\theta \in \R$ (mod $2\pi$). For each $\theta \in (-\pi,\pi]$, let us denote $\sigma_\theta$ as the set of complex numbers $\lambda$ for which there exists a bounded non-trivial solution $v \in L^\infty (\R)$ to the quasi-periodic problem
\begin{equation}
\label{quasip}
\begin{aligned}
\lambda v &= \cL^\ep v,\\
\partial_x^j v(L_\ep) &= e^{i \theta} \partial_x^j v(0), \qquad j = 0,1,2.
\end{aligned}
\end{equation}
The \emph{Floquet spectrum} is defined as
\[
 \sigma_F := \!\!\bigcup_{-\pi<\theta \leq \pi}\sigma_\theta.
\]
It turns out that the Floquet spectrum coincides with $\sigma(\cL^\ep)_{|L^2}$.
\begin{lemma}[Floquet characterization of the spectrum] $\sigma(\cL)_{|L^2} = \sigma_F$.
\end{lemma}
\begin{proof}
See the proof of Proposition 3.4 in Jones \emph{et al.} \cite{JMMP14}.
\end{proof}
%
%
%
%
One advantage of this characterization is that the purely essential spectrum $\sigma(\cL)_{|L^2}$ can be recast as the union of partial spectra $\sigma_\theta$. Moreover, for each fixed $\theta$ the set $\sigma_\theta$ is discrete because it is the zero set of an analytic function (see, e.g., \cite{JMMP14}, p. 4649). Notice that if $\theta = 0$ then the boundary conditions in \eqref{quasip} become periodic and, consequently, $\sigma_0$ detects perturbations 
which are co-periodic. 

One can further remove the dependence on $\theta$ of the boundary conditions in \eqref{quasip} and pose the problem in a proper periodic space independently of (but indexed by) $\theta$, by means of a Bloch-type transformation (cf. \cite{KaPro13}). Indeed, let us define
\begin{equation}
\label{defBloch}
y:= \frac{\pi \xi}{L_\ep}, \qquad w(y) := e^{-i\theta y/\pi} v\Big( \frac{L_\ep y}{\pi}\Big).
\end{equation}
Then, it is easy to show that for any given $\theta \in (-\pi, \pi]$ any solution to \eqref{quasip} is transformed to the solution to the following equation with periodic boundary conditions,
\[
\left\{
\begin{aligned}
\lambda w &= - \frac{1}{L_\ep^3} \big(i \theta + \pi \partial_y\big)^3 w +  \frac{1}{L_\ep^2} \big(i \theta + \pi \partial_y\big)^2 w +  \frac{1}{L_\ep} a_1^\ep(L_\ep y/\pi) \big(i \theta + \pi \partial_y\big) w +\\
&\;\; + a_0^\ep(L_\ep y/\pi) w,\\
\partial_y^j w(\pi) &= \partial_y^j w(0), \qquad j =0,1,2,
\end{aligned}
\right.
\]
where the coefficients $a_j^\ep(\cdot)$ are defined in \eqref{defaeps}. Therefore, multiplying last equation by $L_\ep^3 > 0$ and defining the rescaled spectral parameter as
\[
\widetilde{\lambda} := L_\ep^3 \lambda,
\]
we arrive at a family of spectral problems
\begin{equation}
\label{famBlochsp}
\widetilde{\lambda} w = \cL_\theta^\ep w,
\end{equation}
indexed by $\theta \in (-\pi,\pi]$, where the Bloch-type operators are defined as
\begin{equation}
\label{defBlochops}
\left\{
\begin{aligned}
\cL_\theta^\ep &: \Ldper([0,\pi]) \to \Ldper([0,\pi]), \\
D(\cL_\theta^\ep) &= \Htper([0,\pi]),\\
\cL_\theta^\ep &= - \big(i \theta + \pi \partial_y\big)^3 + L_\ep \big(i \theta + \pi \partial_y\big)^2 + L_\ep^2 \widetilde{a}_1^\ep(y) \big(i \theta + \pi \partial_y\big) + L_\ep^3 \widetilde{a}_0^\ep(y)\Id.
\end{aligned}
\right.
\end{equation}
Here, we have defined
\[
\widetilde{a}_j^\ep(y) := a_j^\ep(L_\ep y/\pi), \qquad y \in [0,\pi], \qquad j=0,1.
\]

Since the family has compactly embedded domains in $\Ldper([0,\pi])$ then their spectra consist entirely of discrete point eigenvalues, $\sigma(\cL^\ep_\theta)_{|\Ldper} = \ptsp(\cL^\ep_\theta)_{|\Ldper}$. In this fashion, we relate the $L^2$-spectrum of the operator in \eqref{defL}, which is purely essential, to the discrete point spectrum of the family of operators \eqref{defBlochops}, defined on a periodic Sobolev space.  Moreover, since for any $\lambda$ in the spectrum we have $\Re \lambda > 0$ if and only if $\Re \widetilde{\lambda} > 0$, we immediately obtain the following criterion for instability.
\begin{proposition}[spectral instability criterion]
\label{propcrit}
For each $\ep \in (0,\ep_0)$ the periodic wave $\varphi^\ep$ is spectrally unstable if and only if there exists $\theta_0 \in (-\pi, \pi]$ for which
\[
\ptsp(\cL_{\theta_0}^\ep)_{|\Ldper} \cap \{ \lambda \in \C \, : \, \Re \lambda > 0\} \neq \varnothing.
\]
\end{proposition}

\section{Spectral instability}
\label{secinstab}

In this section we follow recent stability analyses of periodic waves (cf. \cite{AlPl21,AMP22,CheDu22,KDT19}) to prove our main result, pertaining to the spectral instability of the waves of Theorem \ref{thmexist}.

\begin{theorem}[spectral instability of small amplitude periodic waves]
\label{mainthm}
There exists ${\ep}_1\in(0,\ep_0]$ such that every small-amplitude periodic wave $\varphi^\ep$ with $\ep\in(0,\ep_1)$ from Theorem \ref{thmexist} is spectrally unstable, that is, the Floquet spectrum of the linearized operator around the wave intersects the unstable half plane, namely
\[
\sigma(\cL^\ep)_{|L^2} \cap \{ \lambda \in \C \, : \, \Re \lambda > 0\} \neq \varnothing.
\]
\end{theorem}

In order to prove Theorem \ref{mainthm}, we first need to pose the spectral problem \eqref{famBlochsp} for each operator in \eqref{defBlochops} as a spectral problem for a perturbed operator, where the perturbation parameter will depend on $\ep > 0$.

\subsection{Relatively bounded perturbations}

In this section we recast the spectral equation \eqref{famBlochsp} for each Bloch operator as a perturbation problem. For the forthcoming analysis, it is crucial that these perturbation operators are relatively bounded with respect to the operators of lower order. First, let us observe that, from Theorem \ref{thmexist}, the periodic waves under consideration have fundamental period and amplitude given by
\[
L_\ep = \frac{2\pi}{\sqrt{r}} + O(\ep) =: L_0 + O(\ep),
\]
and
\[
|\varphi^\ep(x)|, |(\varphi^\ep)'(x)| = O(\sqrt{|c(\ep)- c_0|}) = O(\sqrt{\ep}),
\]
respectively, as $\ep \to 0^+$, and uniformly for all $x \in [0, L_\ep]$. Hence, we write
\[
L_\ep = L_0 + \sqrt{\ep} L_1,
\]
where $L_1 := (L_\ep - L_0)/\sqrt{\ep} = O(\sqrt{\ep}) = O(1)$ as $\ep \to 0^+$ and
\[
c = c(\ep) = c_0 + \sqrt{\ep} c_1,
\]
where $c_1 := (c(\ep) - c_0)/\sqrt{\ep} = O(1)$. Note that, even though the speed and the fundamental period are perturbations of $O(\ep)$, for convenience and without loss of accuracy, we regard both quantities as perturbations of $O(\sqrt{\ep})$, which is the size of the amplitude of the waves. Therefore, the first order coefficient of each Bloch operator in \eqref{defBlochops} can be written as
\[
\begin{aligned}
L_\ep^2 \widetilde{a}_1^\ep(y) = L_\ep^2 a_1^\ep(L_\ep y/\pi) &= (L_0 + O(\ep))^2 a_1^\ep(L_\ep y/\pi)\\
&=  (L_0 + O(\ep))^2 (c(\ep) -\alpha \varphi^\ep(L_\ep y/\pi))\\
&= (L_0 + O(\ep))^2 (c_0 + O(\ep) - \alpha O(\sqrt{\ep}))\\
&= L_0^2 c_0 + \sqrt{\ep}b_1(y),
\end{aligned}
\]
where
\[
b_1(y) := \frac{L_\ep^2 a_1^\ep(L_\ep y/\pi)  - L_0^2 c_0}{\sqrt{\ep}} = O(1), \qquad \forall y \in [0,\pi].
\]
Recall that $L_0 = 2\pi / \sqrt{r}$ and $c_0 = -r$, so that $L_0^2 c_0 = - 4\pi^2 < 0$. Likewise, the coefficient of zeroth order is
\[
\begin{aligned}
L_\ep^3 \widetilde{a}_0^\ep(y) = L_\ep^3 a_0^\ep(L_\ep y/\pi) &= (L_0 + O(\ep))^3 a_0^\ep(L_\ep y/\pi)\\
&=  (L_0 + O(\ep))^3 \big(r(1- 2 \varphi^\ep(L_\ep y/\pi)) - \alpha (\varphi^\ep)'(L_\ep y/\pi)\big)\\
&= (L_0 + O(\ep))^3 (r + O(\sqrt{\ep}))\\
&= L_0^3 r + O(\sqrt{\ep})\\
&=  L_0^3 r + \sqrt{\epsilon} b_0(y),
\end{aligned}
\]
where
\[
b_0(y) := \frac{L_\ep^3 a_0^\ep(L_\ep y/\pi)  - L_0^3 r}{\sqrt{\ep}} = O(1), \qquad \forall y \in [0,\pi].
\]
Finally, $L_\ep = L_0 + O(\ep)$ implies that
\[
b_2 := \frac{L_\ep - L_0}{\sqrt{\ep}} = O(1),
\]
is a uniformly bounded constant coefficient (indeed, $b_2 \leq C\sqrt{\ep} \leq C$ for $0 < \ep \leq \ep_0$, sufficiently small). Let us denote
\[
\vep := \sqrt{\ep} \in (0,\sqrt{\ep_0}) = (0, \vep_0). 
\]
To sum up, we recast the Bloch operators \eqref{defBlochops} in perturbation form as
\[
\begin{aligned}
\cL_\theta^\ep &= - \big(i \theta + \pi \partial_y\big)^3 + (L_0 + \vep b_2) \big(i \theta + \pi \partial_y\big)^2 + \\
&\quad + (L_0^2 c_0 + \vep b_1(y))(i \theta + \pi \partial_y) + (L_0^3 r + \vep b_0(y)) \Id,\\
&= \cL_\theta^0 + \vep \cL_\theta^1,
\end{aligned}
\]
for each $\theta \in (-\pi, \pi]$, where we have defined the operators $\cL_\theta^j$, $j=0,1$, as
\begin{equation}
\label{defL0}
\left\{
\begin{aligned}
\cL_\theta^0 &: \Ldper([0,\pi]) \to \Ldper([0,\pi]), \\
D(\cL_\theta^0) &= \Htper([0,\pi]),\\
\cL_\theta^0 &= - \big(i \theta + \pi \partial_y\big)^3 + L_0 \big(i \theta + \pi \partial_y\big)^2 + L_0^2 c_0 \big(i \theta + \pi \partial_y\big) + L_0^3 r \, \Id,
\end{aligned}
\right.
\end{equation}
and
\begin{equation}
\label{defL1}
\left\{
\begin{aligned}
\cL_\theta^1 &: \Ldper([0,\pi]) \to \Ldper([0,\pi]), \\
D(\cL_\theta^1) &= \Hdper([0,\pi]),\\
\cL_\theta^1 &= b_2 \big(i \theta + \pi \partial_y\big)^2 + b_1(y) \big(i \theta + \pi \partial_y\big) + b_0(y) \, \Id,
\end{aligned}
\right.
\end{equation}
for each $\theta \in (-\pi,\pi]$. Notice that $\cL_\theta^0$ is a third order differential operator with constant coefficients, whereas $\cL_\theta^1$ is a second order operator with bounded coefficients. Moreover, both operators are closed and densely defined in $\Ldper([0,\pi])$, with $D(\cL_\theta^0) \subset D(\cL_\theta^1)$.

In this fashion, we have rewritten the spectral problem \eqref{famBlochsp} as a perturbed spectral problem of the form
\begin{equation}
\label{pertsp}
\widetilde{\lambda} w = (\cL_\theta^0 + \vep \cL_\theta^1) w, \qquad w \in \Htper([0,\pi]) = D(\cL_\theta^\ep).
\end{equation}

Before showing that this is a case of relatively bounded perturbations, we need to prove a preliminary lemma.
\begin{lemma}
\label{lemineq}
For every $u \in \Htper([0,\pi])$ there hold the estimates
\begin{align}
\| u_{yy} \|_{\Ldper} &\leq \tfrac{2}{3} \delta^{3/2} \| u_{yyy} \|_{\Ldper} + \tfrac{1}{3} \delta^{-3} \| u \|_{\Ldper}, \label{estuyy}\\
\| u_{y} \|_{\Ldper} &\leq \tfrac{1}{3} \delta^{3} \| u_{yyy} \|_{\Ldper} + \tfrac{2}{3} \delta^{-3/2} \| u \|_{\Ldper}, \label{estuy}
\end{align}
for any arbitrary $\delta > 0$.
\end{lemma}
\begin{proof}
Integrating by parts and using the periodicity of $u \in \Htper([0,\pi])$ and of its derivatives, one obtains
\[
\begin{aligned}
\| u_{y} \|_{\Ldper}^2 &= \int_0^\pi |u_y|^2 \, dy = \left.(\overline{(u_y)} u)\right|_0^\pi - \int_0^\pi \overline{(u_{yy})} u \, dy \leq \int_0^\pi |u_{yy}||u| \, dy \\
&\leq \| u_{yy} \|_{\Ldper} \| u \|_{\Ldper}.
\end{aligned}
\]
Similarly, we have
\[
\| u_{yy} \|_{\Ldper} \leq \| u_{yyy} \|_{\Ldper}^{1/2} \| u_{y} \|_{\Ldper}^{1/2} \leq \| u_{yyy} \|_{\Ldper}^{1/2} \| u_{yy} \|_{\Ldper}^{1/4}  \| u \|_{\Ldper}^{1/4},
\]
which yields, in turn,
\[
\| u_{yy} \|_{\Ldper}^{3/4} \leq \| u_{yyy} \|^{1/2}_{\Ldper} \| u \|^{1/4}_{\Ldper}.
\]
The last inequality is equivalent to
\begin{equation}
\label{aux1}
\| u_{yy} \|_{\Ldper} \leq \| u_{yyy} \|^{2/3}_{\Ldper}  \| u \|^{1/3}_{\Ldper},
\end{equation}
and, as a consequence, 
\begin{align}
\| u_{yy} \|_{\Ldper} \leq \| u_{yy} \|^{1/2}_{\Ldper} \| u \|^{1/2}_{\Ldper} &\leq \| u_{yyy} \|^{1/3}_{\Ldper}  \| u \|^{1/6}_{\Ldper} \| u \|^{1/2}_{\Ldper}\nonumber \\
&=  \| u_{yyy} \|^{1/3}_{\Ldper}  \| u \|^{2/3}_{\Ldper}. \label{aux2}
\end{align}

Next, apply Young's inequality, $ab \leq a^p/p + b^q/q$ with $p^{-1} + q^{-1} =1$ and $a, b > 0$ by taking $a = \delta \|u_{yyy}\|^{2/3}_{\Ldper}$, $b = \delta^{-1} \| u \|_{\Ldper}^{1/3}$, $p = 3/2$, $q =3$ and any $\delta > 0$, to obtain
\[
\| u_{yyy} \|^{2/3}_{\Ldper}  \| u \|^{1/3}_{\Ldper} \leq \tfrac{2}{3} \delta^{3/2} \| u_{yyy} \|_{\Ldper} + \tfrac{1}{3} \delta^{-3} \| u \|_{\Ldper}.
\]
Upon substitution into \eqref{aux1} we get \eqref{estuyy}.
Similarly, take $a = \delta \| u_{yyy} \|_{\Ldper}^{1/3}$, $b = \delta^{-1} \| u \|_{\Ldper}^{2/3}$, $p =3$, $q = 3/2$ and apply Young´s inequality to obtain
\[
\| u_{yyy} \|^{1/3}_{\Ldper}  \| u \|^{2/3}_{\Ldper} \leq \tfrac{1}{3} \delta^3 \| u_{yyy} \|_{\Ldper} + \tfrac{2}{3} \delta^{-3/2} \| u \|_{\Ldper}.
\]
Substitute last inequality into \eqref{aux2} to arrive at \eqref{estuy}. The lemma is proved.
\end{proof}

Let us now recall that if $\cT, \cL : X \to Y$ are linear operators in $X$, $Y$, Banach spaces, we say that $\cT$ is relatively bounded with respect to $\cL$ (or simply $\cL$-bounded) provided that $D(\cL) \subset D(\cT)$ and that there exist constants $a, b \geq 0$ such that
\[
\| \cT u \| \leq a \| u \| + b \| \cL u \|,
\]
for all $u \in D(\cL)$, see Kato \cite{Kat80}, p. 190. The following lemma is the main result of this section.

\begin{lemma}
\label{lemuno}
For any fixed $\theta \in (-\pi, \pi]$, the operator $\cL_\theta^1$ is relatively bounded with respect to $\cL_\theta^0$.
\end{lemma}
\begin{proof}
Since $D(\cL_\theta^0) = \Htper([0,\pi]) \subset \Hdper([0,\pi]) = D(\cL_\theta^1)$, we only need to show that there exist non-negative constants, $a, b \geq 0$, such that
\begin{equation}
\label{Lboundest}
\| \cL_\theta^1 u \|_{\Ldper} \leq a \| u \|_{\Ldper} + b \| \cL_\theta^0 u \|_{\Ldper},
\end{equation}
for each $u \in \Htper([0,\pi])$. First, let us estimate
\[
\begin{aligned}
\| \cL_\theta^1 u \|_{\Ldper} &= \left\| b_2 (i \theta + \pi \partial_y)^2 u + b_1(y) \big(i \theta + \pi \partial_y\big) + b_0(y) u \right\|_{\Ldper}\\
&\leq \pi^2 |b_2| \| u_{yy} \|_{\Ldper} + \pi \big( \| b_1 \|_{L^\infty} + 2 |\theta||b_2| \big) \| u_y \|_{\Ldper} +\\
&\quad + \big( \| b_0 \|_{L^\infty} + |\theta| \| b_1 \|_{L^\infty} + \theta^2 |b_2| \big) \| u \|_{\Ldper}\\
&\leq C_2 \| u_{yy} \|_{\Ldper} + C_1 \| u_{y} \|_{\Ldper} + C_0 \| u \|_{\Ldper},
\end{aligned}
\]
with uniform constants
\[
C_2 := \pi^2 |b_2|, \quad C_1 := \pi \big( \| b_1 \|_{L^\infty} + 2 |\theta||b_2| \big), \quad C_0 := \big( \| b_0 \|_{L^\infty} + |\theta| \| b_1 \|_{L^\infty} + \theta^2 |b_2| \big),
\]
because $|\theta| \leq \pi$. Using estimates \eqref{estuyy} and \eqref{estuy} we arrive at
\begin{equation}
\label{estL1theta}
\| \cL_\theta^1 u \|_{\Ldper} \leq C_{3,\delta} \|u_{yyy} \|_{\Ldper} + C_{4,\delta} \| u\|_{\Ldper},
\end{equation}
where
\[
C_{3,\delta} := \tfrac{2}{3} C_2 \delta^{3/2} + \tfrac{1}{3} C_1 \delta^3, \quad C_{4,\delta} := \tfrac{1}{3} C_2 \delta^{-3} + \tfrac{2}{3} C_1 \delta^{-3/2} + C_0,
\]
are uniform positive constants for each $\delta > 0$. 

On the other hand, we have the estimate,
\[
\begin{aligned}
\| \cL_\theta^0 u \|_{\Ldper} &= \left\| - \big(i \theta + \pi \partial_y\big)^3 u + L_0 \big(i \theta + \pi \partial_y\big)^2 u + L_0^2 c_0 \big(i \theta + \pi \partial_y\big) u + L_0^3 r  u \right\|_{\Ldper}\\
&=  \left\| - \pi^3 u_{yyy} + (L_0 \pi^2 - 3 i \theta \pi^2) u_{yy} + \big(2i\theta \pi L_0 + L_0^2 c_0 \pi - 3(i\theta)^2 \pi\big)u_y + \right.\\
& \qquad \left. + \big(L_0 (i\theta)^2 + L_0^2 c_0 i \theta + rL_0^3 - (i\theta)^3\big)u  \right\|_{\Ldper}\\
&\geq \pi^3 \| u_{yyy} \|_{\Ldper} - |b_3| \| u_{yy} \|_{\Ldper} - |b_4| \| u_y \|_{\Ldper} - |b_5| \| u \|_{\Ldper},
\end{aligned}
\]
where
\[
\begin{aligned}
b_3 &:= L_0 \pi^2 - 3i\theta \pi^2,\\
b_4 &:= 2i\theta \pi L_0 + L_0^2 c_0 \pi - 3(i\theta)^2 \pi,\\
b_5 &:= L_0 (i\theta)^2 + L_0^2 c_0 i \theta + rL_0^3 - (i\theta)^3,
\end{aligned}
\]
are uniformly bounded constants. Now, use estimates \eqref{estuyy} and \eqref{estuy} to obtain
\[
\begin{aligned}
\| \cL_\theta^0 u \|_{\Ldper} &\geq \pi^3 \| u_{yyy} \|_{\Ldper} - |b_3| \big( \tfrac{2}{3} \delta^{3/2} \| u_{yyy} \|_{\Ldper} + \tfrac{1}{3} \delta^{-3} \| u \|_{\Ldper} \big) \\
& \qquad - |b_4| \big( \tfrac{1}{3} \delta^{3} \| u_{yyy} \|_{\Ldper} + \tfrac{2}{3} \delta^{-3/2} \| u \|_{\Ldper} \big)\\
& \qquad - |b_5| \| u \|_{\Ldper}\\
&\geq \big( \pi^3 - \tfrac{2}{3} |b_3| \delta^{3/2} - \tfrac{1}{3} |b_4| \delta^3 \big) \| u_{yyy} \|_{\Ldper} \\
& \qquad - \big( |b_5| + \tfrac{2}{3} \delta^{-3/2} |b_4| + \tfrac{1}{3} |b_3| \delta^{-1/3} \big) \| u \|_{\Ldper}.
\end{aligned}
\]
Now, choose $0 < \delta \ll 1$ sufficiently small such that
\[
C_{5,\delta} := \pi^3 - \tfrac{2}{3} |b_3| \delta^{3/2} - \tfrac{1}{3} |b_4| \delta^3 \geq \tfrac{1}{3} \pi^3 > 0,\\
\]
and set
\[
C_{6,\delta} := |b_5| + \tfrac{2}{3} \delta^{-3/2} |b_4| + \tfrac{1}{3} |b_3| \delta^{-1/3} > 0.
\]
This implies that
\[
\| u_{yyy} \|_{\Ldper} \leq \frac{1}{C_{5,\delta}} \| \cL_\theta^0 u \|_{\Ldper} + \frac{C_{6,\delta}}{C_{5,\delta}} \| u \|_{\Ldper}.
\]
Substituting into \eqref{estL1theta}, we obtain
\[
\begin{aligned}
\| \cL_\theta^1 u \|_{\Ldper} &\leq \frac{C_{3,\delta}}{C_{5,\delta}}  \| \cL_\theta^0 u \|_{\Ldper}  + \Big( \frac{C_{3,\delta} C_{6,\delta}}{C_{5,\delta}} + C_{4,\delta}\Big) \| u \|_{\Ldper}\\
&= b \| \cL_\theta^0 u \|_{\Ldper} + a \| u \|_{\Ldper},
\end{aligned}
\]
as claimed, with $a = {C_{3,\delta}}/{C_{5,\delta}} > 0$ and $b = C_{4,\delta} + C_{3,\delta} C_{6,\delta} / C_{5,\delta} > 0$. The lemma is proved.
\end{proof}

\subsection{A quick review of perturbation theory for linear operators}

For completeness, in this section we briefly recall the basic perturbation theory for a family of linear operators of the form $\cL(\kappa) = \cL^{(0)} + \kappa \cL^{(1)} + \kappa^2 \cL^{(2)} + \ldots$, where $\kappa$ is a complex parameter. The reader is referred to the classical books by Hislop and Sigal \cite{HiSi96} (chapter 15) and Kato \cite{Kat80} (chapter VII) for more information. We describe the fundamental conditions under which an eigenvalue $\lambda_0$ of $\cL^{(0)}$ persists for $\kappa$ in a neighborhood of the origin.  Let us first recall the definition of an analytic family of type (A) (see \cite{Kat80}, p. 375, or \cite{HiSi96}, \S 15.4).

\begin{definition}
A family of linear closed operators $\cL(\kappa) : X \to Y$, with $X, Y$ Banach spaces, defined for $\kappa$ in an open complex domain $B \subset \C$ containing the origin, is said to be \emph{analytic of type (A)} if the domains are independent of $\kappa$, that is, $D(\cL(\kappa)) \equiv D$ for all $\kappa \in B$, and the mapping $\kappa \mapsto \cL(\kappa) u$ is analytic in $\kappa \in B$ for each $u \in D$. In that case, $\cL(\kappa)$ has a convergent Taylor expansion of the form
\[
\cL(\kappa) u = \cL^{(0)} u + \kappa \cL^{(1)} u + \kappa^2 \cL^{(2)} u + \ldots,
\]
converging in a disk $|\kappa| < R$ inside $B$ independent of $u$ for some radius $R > 0$. Here $\cL(0) = \cL^{(0)}$ and $\cL^{(j)} : D \subset X \to Y$ are linear operators for each $j \in \N$.
\end{definition}

The following result is the general stability theorem for discrete eigenvalues of analytic perturbations of a closed operator. Essentially it says that a discrete eigenvalue splits into several eigenvalue branches under a perturbation. These branches tend to the original discrete eigenvalue when the perturbation parameter goes to zero.

\begin{theorem}
\label{pertthm}
Let $\cL(\kappa)$ be an analytic family of type (A) about $\kappa_0 = 0$. Let $\lambda_0$ be a discrete eigenvalue of $\cL^{(0)} = \cL(0)$. Then there exist families $\lambda_\ell(\kappa)$, $\ell = 1, \ldots, q$, of discrete eigenvalues of $\cL(\kappa)$ such that:
\begin{itemize}
\item[(i)] $\lambda_\ell(0) = \lambda_0$ for all $\ell = 1, \ldots, q,$ and the total algebraic multiplicity of the eigenvalues $\lambda_\ell(\kappa)$, $\ell = 1, \ldots, q,$ is equal to the algebraic multiplicity of $\lambda_0$;
\item[(ii)] each family $\lambda_\ell(\kappa)$ is analytic in $\kappa^{1/p}$ for some positive integer $p \in \Z^+$, that is, each family has a Puiseux series expansion of the form
\[
\lambda_\ell(\kappa) = \lambda_0 + \beta_1 \omega^\ell \kappa^{1/p} + \beta_2 \omega^{2\ell} \kappa^{2/p} + \ldots,
\]
for some $\beta_j \in \C$ and where $\omega = e^{2 \pi i/p}$, converging in a (possibly smaller) disk $|\kappa| < R' \leq R$. In particular, the branches are continuous in $\kappa$ near the origin and $\lambda_\ell(\kappa) \to \lambda_0$ when $\kappa \to 0$.
\end{itemize}
\end{theorem}
\begin{proof}
See Theorem 15.11, p. 155 in \cite{HiSi96}. There are no negative powers of $\kappa^{1/p}$ so that the branches are continuous at the origin, see Knopp \cite{Knopp}.
\end{proof}

In order to apply the previous perturbation result, let us first examine the perturbed spectral problem \eqref{pertsp} for the Bloch operator with $\theta = 0$, namely, to
\begin{equation}
\label{pertsp2}
\widetilde{\lambda} w = (\cL_0^0 + \vep \cL_0^1) w, \quad \qquad w \in \Htper([0,\pi]),
\end{equation}
and consider its complexification:
\begin{equation}
\label{compf}
\left\{
\begin{aligned}
\widetilde{\cL}_0(\kappa) &:= \cL_0^0 + \kappa \cL_0^1,\\
D(\widetilde{\cL}_0(\kappa)) &= \Htper([0,\pi]), \\
\widetilde{\cL}_0(\kappa) &: \Ldper([0,\pi]) \to \Ldper([0,\pi]),
\end{aligned}
\right.
\end{equation}
where $\kappa \in B_{\hat{\vep}} = \{\kappa \in \C \, : \, |\kappa| < \hat{\vep}\}$, for some $\hat{\vep} > 0$ to be determined. That is, we extend the perturbed spectral problem \eqref{pertsp2} to a complex open neighborhood of the origin. The family of operators in \eqref{compf} has a common domain, $D:= D(\widetilde{\cL}_0(\kappa)) \equiv \Htper([0,\pi])$, which is independent of $\kappa \in B_{\hat{\vep}}$. Notice that the Bloch operators for $\theta = 0$ are given by
\begin{equation}
\label{defL00}
\left\{
\begin{aligned}
\cL_0^0 &: \Ldper([0,\pi]) \to \Ldper([0,\pi]), \\
D(\cL_0^0) &= \Htper([0,\pi]),\\
\cL_0^0 &= - \pi^3 \partial_y^3 + L_0 \pi^2 \partial_y^2 + L_0^2 c_0 \pi \partial_y + L_0^3 r \, \Id,
\end{aligned}
\right.
\end{equation}
and
\begin{equation}
\label{defL10}
\left\{
\begin{aligned}
\cL_0^1 &: \Ldper([0,\pi]) \to \Ldper([0,\pi]), \\
D(\cL_0^1) &= \Hdper([0,\pi]),\\
\cL_0^1 &= b_2 \pi^2 \partial_y^2 + b_1(y) \pi \partial_y + b_0(y) \, \Id.
\end{aligned}
\right.
\end{equation}

We start by making a straightforward observation.
\begin{lemma}
\label{lemdos}
The Bloch operator $\cL_0^0$ has a real positive discrete eigenvalue, $\widetilde{\lambda}_0 = r L_0^3 > 0$, which is associated to the constant eigenfunction $w_0(y) \equiv 1/ \sqrt{\pi} \in \Htper([0,\pi])$, with $\|w_0\|_{\Ldper} = 1$. 
\end{lemma}
\begin{proof}
Follows from a direct computation, as $\cL_0^0 w_0 = r L_0^3 w_0$. Moreover, clearly $w_0 \in \Htper([0,\pi])$ and $\| w_0 \|_{\Ldper}^2 = \int_0^\pi |w_0(y)|^2 \, dy = 1$. As we have already mentioned, all eigenvalues of the Bloch operators \eqref{defBlochops} are discrete. The lemma is proved.
\end{proof}

Next, we show that the family $\widetilde{\cL}_0(\kappa)$ is an analytic family of type (A). This is a direct consequence of $\cL_0^1$ being $\cL_0^0$-bounded.
\begin{lemma}
\label{lemtres}
There exists $\hat{\vep} > 0$ such that the family of operators \eqref{compf} is an analytic family of type (A) for $\kappa \in B_{\hat{\vep}}$.
\end{lemma}
\begin{proof}
From Lemma \ref{lemuno}, in the particular case with $\theta  = 0$ we already know that the operator $\cL_0^1$ is relatively bounded with respect to $\cL_0^0$ on the space $\Ldper([0,\pi])$. Therefore, there exist $a, b > 0$ such that
\[
\| \cL_0^1 u \|_{\Ldper} \leq a \| u \|_{\Ldper} + b \| \cL_0^0 u \|_{\Ldper},
\]
for all $u \in D(\cL_0^0) = \Htper([0,\pi]) \subset D(\cL_0^1) = \Hdper([0,\pi])$. Thus, we can apply Theorem VII-2.6 and Remark VII-2.7 in Kato \cite{Kat80}, pp. 377-378, to conclude that there exists a radius $\hat{\vep} > 0$ such that the family $\widetilde{\cL}_0(\kappa) = \cL_0^0 + \kappa \cL_0^1$ is analytic of type (A) for $|\kappa| < \hat{\vep}$ (in fact, $\hat{\vep} < b^{-1}$). 
\end{proof}

After these preparations, we are able to apply perturbation theory for the family of linear operators.

\begin{lemma}
\label{lemcuatro}
There exist $\vep_1 > 0$ and families of discrete eigenvalues, $\widetilde{\lambda}_\ell(\kappa)$, $\ell =1, \ldots, q$, of the operators $\widetilde{\cL}_0(\kappa) = \cL_0^0 + \kappa \cL_0^1$ for $|\kappa| < \vep_1$ such that $\widetilde{\lambda}_\ell(0) = \widetilde{\lambda}_0 = r L_0^3 > 0$. Moreover, for each $0 < \vep < \vep_1$ and $|\kappa|< \vep$ there holds
\begin{equation}
\label{intersecti}
\ptsp(\cL_0^0 + \kappa \cL_0^1)_{|\Ldper} \cap \{ \lambda \in \C \, : \, |\lambda - r L_0^3| < \zeta(\vep)\} \neq \varnothing,
\end{equation}
for some radius $0 < \zeta(\vep) = O(\vep^{1/p})$ and some $p \in \Z_+$.
\end{lemma}
\begin{proof}
From Lemmata \ref{lemtres} and \ref{lemdos} we know that $\cL_0^0 + \kappa \cL_0^1$ is an analytic family of type (A) for $|\kappa| < \hat{\vep}$ and that $\widetilde{\lambda}_0 = r L_0^3 > 0$ is a discrete eigenvalue of $\cL_0^0$. Henceforth, we apply Theorem \ref{pertthm} to conclude that there exist families of discrete eigenvalues, $\widetilde{\lambda}_\ell(\kappa)$, $\ell =1, \ldots, q$, of the operators $\widetilde{\cL}_0(\kappa) = \cL_0^0 + \kappa \cL_0^1$, for $|\kappa| < \hat{\vep}$ such that $\widetilde{\lambda}_\ell(0) = \widetilde{\lambda}_0 = r L_0^3 > 0$ and the families have Puiseux series expansions of the form
\[
\widetilde{\lambda}_\ell(\kappa) = \widetilde{\lambda}_0 + \beta_1 \omega^\ell \kappa^{1/p} + \beta_2 \omega^{2\ell} \kappa^{2/p} + \ldots,
\]
with $\omega = e^{2 \pi i/p}$ for some $p \in \Z$, $p > 0$, converging in a (possibly smaller) disk $|\kappa| < \vep_1 \leq \hat{\vep}$. Each $\widetilde{\lambda}_\ell(\kappa)$ is a discrete eigenvalue of $\widetilde{\cL}_0(\kappa) = \cL_0^0 + \kappa \cL_0^1$ with $\widetilde{\lambda}_\ell(0) = \widetilde{\lambda}_0 = r L_0^3 > 0$ for each $\ell$. Property \eqref{intersecti} follows from continuity of the expansion and that $|\widetilde{\lambda}_\ell(\kappa) - \widetilde{\lambda}_0| \leq C |\kappa|^{1/p}$ for $|\kappa|$ small. This proves the result.
\end{proof}

We now prove the main result of the paper.

\subsection{Proof of Theorem \ref{mainthm}}
In view that $\vep = \sqrt{\ep}$, let us choose $0 < \ep_1 := \min \{\ep_0, \vep_1^2\}$. Hence, from Lemma \ref{lemcuatro} we conclude that there exists $\theta_0 = 0 \in (-\pi,\pi]$ for which the spectrum of the perturbed Bloch operator with $\theta = 0$ and $\ep \in (0, \ep_1)$ intersects the unstable half plane (in view that $\widetilde{\lambda}_0 =  r L_0^3> 0$):
\[
\ptsp(\cL_{0}^\ep)_{|\Ldper} \cap \{ \lambda \in \C \, : \, \Re \lambda > 0\} \neq \varnothing.
\]
Therefore, from the spectral instability criterion stated in Proposition \ref{propcrit} we obtain the result. Notice that if we vary $\theta \approx 0$ (within a small neighborhood of the origin), we obtain actual curves of Floquet spectrum that remain in a neighborhood of the unstable half plane. This completes the proof of Theorem \ref{mainthm}. 
\qed

\section{Discussion}

In this paper, we have proved the existence of bounded periodic waves for the KdVBF equation \eqref{KdVBFn}. For that purpose we have verified that, for any fixed positive values of the physical parameters $\alpha > 0$ and $r > 0$, a family of small-amplitude and finite period waves emerges from a subcritical local Hopf bifurcation around a critical value of the wave speed given by $c_0 = -r$. To illustrate the emergence of the waves we conducted numerical approximations of both the waves and of the flow in phase space for fixed parameter values. In addition, we studied the stability of the family of periodic waves as solutions to the PDE. In particular, we proved that these waves are spectrally unstable or, in other words, that the Floquet spectrum of the linearized operator intersects the unstable complex half plane. Heuristically, this instability result has a direct interpretation: the small amplitude waves shrink to the zero equilibrium solution when the small parameter tends to zero, $\ep \to 0$, and the linearized operator around the wave tends (formally) to a constant coefficient linearized Bloch operator around the origin. The latter has a spectrum determined by a dispersion relation that intersects the unstable half plane thanks to the positive sign of the reaction term at the rest state. We then invoke the classical perturbation theory for linear operators to conclude. 

This perturbative technique has been recently applied to study stability of small-amplitude waves for scalar viscous balance laws \cite{AlPl21}, scalar Hamiltonian equations \cite{KDT19}, hyperbolic systems of balance laws \cite{AMP22} and scalar reaction-diffusion equations coupled to one ordinary differential equation \cite{CheDu22}. This appears to be a general behavior and the method could be used to study the instability of small periodic waves arising in other contexts. In a recent contribution, Chen and Duan \cite{CheDu23} found the conditions (in a unified framework) to prove the perturbation results on the related spectra for a large class of second order differential operators with asymptotically constant coefficients. Their result does not apply, however, to the KdVBF equation, which is of third order. From a technical viewpoint, it is to be noticed that we did not prove that the unstable eigenvalue of the unperturbed operator is simple or non-degenerate. Nonetheless, we directly applied perturbation theory for a higher multiplicity eigenvalue for which the perturbation eigenvalue branches can be represented by a Pusieux series which remain continuous in the limit. This allows us to conclude the result. This feature exhibits the flexibility of perturbation theory of linear operators as a technique to examine instabilities of approximated waves in general situations. Finally, the instability results of the present study may serve as the theoretical basis to demonstrate the orbital (nonlinear) instability of the waves under consideration, as it happens in many other physical systems and model equations such as the KdV equation \cite{LopO02}, viscous balance laws \cite{AAP24}, the critical KdV and NLS models \cite{AngNat09}, KdV systems \cite{AnLN08}, and for general dispersive models \cite{AngNat16}, just to mention a few. 

\section*{Acknowledgements}

The authors are grateful to two anonymous reviewers, whose comments, suggestions and careful revisions greatly improved the quality of the paper. The authors also thank Renato Calleja and Carlos Garc\'{\i}a-Azpeitia for useful conversations about Hopf bifurcation. The work of R. Folino was partially supported by DGAPA-UNAM, program PAPIIT, grant IA-102423. A. Naumkina was fully supported by CONAHCyT (Mexico), through a scholarship for graduate studies (CVU 1081529). The work of R. G. Plaza was partially supported by DGAPA-UNAM, program PAPIIT, grant IN-104922.

\appendix
\section{Calculation of the first Lyapunov coefficient}\label{apendice}

In this section we compute the first Lyapunov coefficient for the linearization of system \eqref{firstordsyst},
\[
\Phi' = F(\Phi) = \begin{pmatrix} \phi_2 \\ \phi_3 \\ \phi_3 + r\phi_1(1-\phi_1) - \alpha \phi_1 \phi_2 + c \phi_2 \end{pmatrix} =: \begin{pmatrix} F_1(\Phi) \\F_2(\Phi) \\ F_3(\Phi) \end{pmatrix},
\]
evaluated at the equilibrium point $\Phi = P_0 = (0,0,0)$ and at the critical value of the bifurcation parameter, $c = c_0 = -r$.

For large systems in dimension $n > 2$ the calculation of the first Lyapunov coefficient, denoted as $a_0 = \ell_1(0)$, can be quite cumbersome. Yet, in our case we use an expression of $\ell_1(0)$ which does not require to compute the normal form of the system nor to transform it into its eigenbasis, and which expresses the coefficient in terms of the original field and its derivatives (see Kuznetsov \cite{Kuz23-4e}, formula (5.34), p. 197):
\begin{equation}
\label{forlyap}
\begin{aligned}
a_0 = \ell_1(0) = \frac{1}{2 \omega_0} \Re \Big[ &\big\langle p, C(q,q,\bq) \big\rangle - 2 \big\langle p, B(q, A_0^{-1} B(q, \bq)) \big \rangle + \\
& + \big\langle p, B(\bq, (i2\omega_0 I - A_0)^{-1} B(q,q)) \big\rangle \Big],
\end{aligned}
\end{equation}
where $\langle p, q \rangle = \sum_{j=1}^3 \bp_j q_j$ is the standard scalar product in $\C^3$ (linear with respect to the the second argument), $\omega_0 = \lambda_{2,3}(c_0) = \rr$,
\[
A_0 = (D_{\Phi} F)_{|(P_0,c_0)} = \begin{pmatrix} 0 & 1 & 0 \\ 0 & 0 & 1 \\ r & -r  & 1 \end{pmatrix},
\]
$q \in \C^3$ is the complex eigenvector of $A_0$ corresponding to the eigenvalue $\lambda_2(c_0) = i \omega_0 = i \rr$, $p \in \C^3$ is the adjoint eigenvector with $A_0^\top p = - i \omega_0 p$, and normalized such that $\langle p, q \rangle = 1$. Finally, for any $x, y, z \in \C^3$, $B = B(x,y) \in \C^3$ and $C = C(x,y,z) \in \C^3$ are the multilinear functions defined as
\[
\begin{aligned}
B_i(x,y) = \sum_{j,k = 1}^3 \big( \partial_{jk}^2 F_i \big)_{|P_0} x_j y_k, \qquad C_i(x,y,z) = \sum_{j,k,l = 1}^3 \big( \partial_{jkl}^3 F_i \big)_{|P_0} x_j y_k z_l,
\end{aligned}
\]
for $1 \leq i \leq 3$.

By inspection, it is clear that $\partial_{jk}^2 F_i = 0$ and $\partial_{jkl}^3 F_i = 0$ for all $1 \leq j,k,l \leq 3$ and for $i = 1,2$. Moreover, $\partial_{13}^2 F_3=\partial_{22}^2 F_3 = \partial_{23}^2 F_3 = \partial_{31}^2 F_3=\partial_{32}^2 F_3=\partial_{33}^2 F_3 = 0$ and the only surviving higher order derivatives are
\[
\partial_{11}^2 F_3 = - 2r, \quad \partial_{12}^2 F_3 = \partial_{21}^2 F_3 = - \alpha.
\]
This implies that for any $x,y,z \in \C^3$, $C_i(x,y,z) \equiv 0$, for all $1 \leq i \leq 3$, while $B_i(x,y) = 0$, for $i = 1,2$ and the only surviving term is
\[
B_3(x,y) = - 2r x_1 y_1 - \alpha x_1 y_2 - \alpha x_2 y_1.
\]
It is not hard to verify that the inverse of $A_0$ is
\[
A_0^{-1} = \begin{pmatrix} 1 & -1/r & 1/r \\ 1 & 0 & 0 \\ 0 & 1 & 0 \end{pmatrix},
\]
and that the normalized eigenvectors are
\[
q = \begin{pmatrix} 1 \\ i \rr \\ -r \end{pmatrix}, \qquad p = \frac{1}{2 \rr (1 + i \rr) } \begin{pmatrix} \rr \\ - \rr + i \\ -i  \end{pmatrix}.
\]
As a consequence, $B_3(q, \bq) = -2r$ and $B_3(q,q) = -2(r + i \alpha \rr)$. Denoting 
\[
s := A_0^{-1} B(q, \bq) = \begin{pmatrix} 1 & -1/r & 1/r \\ 1 & 0 & 0 \\ 0 & 1 & 0 \end{pmatrix} \begin{pmatrix} 0 \\ 0 \\ B_3(q, \bq) \end{pmatrix} = 
\begin{pmatrix} -2 \\ 0 \\ 0  \end{pmatrix},
\]
we obtain $B_3(q, s) = 4r + i 2\alpha \rr$ and, after some algebra, we arrive at
\[
\begin{aligned}
- 2 \big\langle p, B(q, A_0^{-1} B(q, \bq)) \big \rangle &= -2 \langle p, B(q,s) \rangle = - 2 \bp_3 B_3(q,s) \\
&= (1+r)^{-1} \big[ 2(2r+ \alpha) -i 2 \rr(2-\alpha)\big].
\end{aligned}
\]

Therefore, we denote
\begin{equation}
\label{launo}
\varrho_1 := \Re \Big[ - 2 \big\langle p, B(q, A_0^{-1} B(q, \bq)) \big \rangle \Big] = \frac{2(2r+\alpha)}{1+r}.
\end{equation}

Computing further, we observe that
\[
\begin{aligned}
J &:= i2\omega_0 I - A_0 = i2 \rr I - A_0 =\begin{pmatrix} i2\rr & -1 & 0 \\ 0& i2\rr & -1 \\ -r & r & i2 \rr -1 \end{pmatrix},\\
J^{-1} &= \big( 3r (1 - i2 \rr)\big)^{-1} \begin{pmatrix} -3r -i 2\rr & -1 + i2 \rr & 1 \\ r & -4r -i2\rr & i2\rr \\ i2r\rr & r - i 2r \rr & -4r
\end{pmatrix},
\end{aligned}
\]
so that, after some algebraic calculations, we arrive at
\[
m:= J^{-1} B(q,q) = J^{-1} \begin{pmatrix} 0 \\ 0 \\ B_3(q,q) \end{pmatrix} = \Theta \begin{pmatrix} 1 \\ i2 \rr \\ -4r \end{pmatrix},
\]
with
\[
\Theta:= - \frac{2}{3r} \Big( \frac{r + i \alpha \rr}{1 - i 2 \rr}\Big)  = - \frac{2}{3\rr} (1+4r)^{-1} \big( \rr(1-2\alpha) + i(\alpha + 2r)\big).
\]
Whence,
\[
B_3(\bq,m) = -2r \bq_1 m_1 - \alpha \bq_1 m_2 - \alpha \bq_2 m_1 = - (2r + i \alpha \rr) \Theta,
\]
and we can compute
\[
\big\langle p, B(\bq, (i2\omega_0 I - A_0)^{-1} B(q,q)) \big\rangle = \langle p, B(\bq,m) \rangle = \bp_3 B_3(\bq,m) = \frac{- i (2r + i \alpha \rr) \Theta}{2 \rr (1-i \rr)},
\]
yielding, after some algebra,
\begin{equation}
\label{lados}
\begin{aligned}
\varrho_2 &:= \Re \big\langle p, B(\bq, (i2\omega_0 I - A_0)^{-1} B(q,q)) \big\rangle \\
&= \frac{\tfrac{1}{2}\alpha + r}{1+r} \Re \Theta - \frac{(\tfrac{1}{2}\alpha -1)\rr}{1+r} \Im \Theta \\
&=\frac{2(\alpha-1)(\tfrac{1}{2} \alpha + r)}{(1+r)(1 + 4r)}.
\end{aligned}
\end{equation}

In view that $C(q,q,\bq) = 0$, substitution of \eqref{launo} and \eqref{lados} into \eqref{forlyap} yields
\[
\ell_1(0) = \frac{1}{2 \rr} (\varrho_1 + \varrho_2) = \frac{(\tfrac{1}{2}\alpha + 1)(1 + \alpha + 8r)}{(1+r)(1+4r)} > 0.
\]

\def\cprime{$'\!\!$} \def\cprimel{$'\!$}





\end{document}